\documentclass[a4paper,12pt,leqno]{article}

\usepackage[utf8]{inputenc}
\usepackage[T1]{fontenc}
\usepackage{lmodern}
\usepackage{amsmath}
\usepackage{amsthm}
\usepackage{amssymb}
\usepackage{mathrsfs}
\usepackage{hyperref}

\newcommand{\N}{\mathbb N}

\newcommand{\C}{\mathbb C}

\newcommand{\Hil}{\mathscr H}
\newcommand{\Cpt}{\mathscr K}

\newcommand{\sspan}{\ensuremath{\operatorname{span}}}
\newcommand{\linbeg}{\boldsymbol{\operatorname{B}}}%{\mathbb B}
\newcommand{\cb}{\mathrm{cb}}

\newcommand{\dd}{\mathrm d}

\newcommand{\ELL}{{L}}

\newcommand{\ip}[2]{\langle#1,#2\rangle}

\newcommand{\seta}[1]{\ensuremath{\{#1\}}}
\newcommand{\setw}[2]{\setaw{#1}{#2}}%{\seto{#1\:\big|\:#2}}
\newcommand{\setaw}[2]{\seta{#1\::\:#2}}
\newcommand{\Email}{\begingroup \def\UrlLeft{<}\def\UrlRight{>} \urlstyle{tt}\Url}
\newcommand{\mailto}[1]{\href{mailto:#1}{\Email{#1}}}

\newcommand{\contrib}[3]{#1\quad\mailto{#2}{\small\\\quad\textit{#3}}\\[1ex]}

\newcommand{\FA}{A}

\newcommand{\FSA}{B}

\newcommand{\cc}{c}
\newcommand{\G}{G}

\newcommand{\Y}{Y}

\newcommand{\MA}{MA}
\newcommand{\MoA}{M_0A}

\swapnumbers
\numberwithin{equation}{section}
\theoremstyle{plain}
\newtheorem{theorem}{Theorem}[section]
\newtheorem{maintheorem}[theorem]{Theorem}
\newtheorem{proposition}[theorem]{Proposition}
\newtheorem{lemma}[theorem]{Lemma}
\newtheorem{corollary}[theorem]{Corollary}

\theoremstyle{definition}

\theoremstyle{remark}
\newtheorem{remark}[theorem]{Remark}

\renewcommand{\Re}{\mathrm{Re}}
\renewcommand{\Im}{\mathrm{Im}}
\renewcommand{\contrib}[2]{#1\quad\mailto{#2}}

\begin{document}
	\hyphenation{hil-bert-space fran-ces-co lo-rentz boun-ded func-tions haa-ge-rup pro-per-ties a-me-na-bi-li-ty re-wri-ting pro-per-ty tro-els steen-strup jen-sen de-note as-su-ming Przy-by-szew-ska ma-jo-ri-za-tion o-pe-ra-tors mul-ti-pli-ers i-so-mor-phism la-place bel-tra-mi re-pre-sen-ta-tions pro-po-si-tion the-o-rem lem-ma co-rol-la-ry}

	\title{Herz--Schur Multipliers and Non-Uniformly Bounded Representations of Locally Compact Groups}
\author{Troels Steenstrup\thanks{Partially supported by the Ph.D.-school OP--ALG--TOP--GEO.}}
\date{}
\maketitle

	\begin{abstract}
		Let $\G$ be a second countable, locally compact group and let $\varphi$ be a continuous Herz--Schur multiplier on $\G$. Our main result gives the existence of a (not necessarily uniformly bounded) strongly continuous representation $\pi$ of $\G$ on a Hilbert space $\Hil$, together with vectors $\xi,\eta\in\Hil$, such that $\varphi(y^{-1}x)=\ip{\pi(x)\xi}{\pi(y^{-1})^*\eta}$ for $x,y\in\G$ and $\sup_{x\in\G}\|\pi(x)\xi\|\cdot\sup_{y\in\G}\|\pi(y^{-1})^*\eta\|=\|\varphi\|_{\MoA(\G)}$. Moreover, we obtain control over the growth of the representation in the sense that $\|\pi(g)\|\leq\exp(\tfrac{c}{2}\dd(g,e))$ for $g\in\G$, where $e\in\G$ is the identity element, $c$ is a constant and $\dd$ is a metric on $\G$.

	\end{abstract}
	\section*{Introduction}
\label{introCoeff}
Let $\Y$ be a non-empty set. A function $\psi:\Y\times\Y\to\C$ is called a \emph{Schur multiplier} if for every operator $A=(a_{x,y})_{x,y\in\Y}\in\linbeg(\ell^2(\Y))$ the matrix $(\psi(x,y)a_{x,y})_{x,y\in\Y}$ again represents an operator from $\linbeg(\ell^2(\Y))$ (this operator is denoted by $M_\psi A$). If $\psi$ is a Schur multiplier it follows easily from the closed graph theorem that $M_\psi\in\linbeg(\linbeg(\ell^2(\Y)))$, and one refers to $\|M_\psi\|$ as the \emph{Schur norm} of $\psi$ and denotes it by $\|\psi\|_S$.

Let $G$ be a locally compact group. %\input{/Users/troelssj/Documents/Personal/PhD/Thesis/HSintro}
In~\cite{Her:UneGeneralisationDeLaNotionDetransformeeDeFourier-Stieltjes}, Herz introduced a class of functions on $\G$, which was later denoted the class of \emph{Herz--Schur multipliers} on $\G$. By the introduction to~\cite{BF:Herz-SchurMultipliersAndCompletelyBoundedMultipliersOfTheFourierAlgebraOfALocallyCompactGroup}, a continuous function $\varphi:\G\to\C$ is a Herz--Schur multiplier if and only if the function
\begin{equation}
	\label{new0.3}
	\hat\varphi(x,y)=\varphi(y^{-1}x)\qquad(x,y\in\G)
\end{equation}
is a Schur multiplier, and the \emph{Herz--Schur norm} of $\varphi$ is given by
\begin{equation*}
	\|\varphi\|_{HS}=\|\hat\varphi\|_S.
\end{equation*}

In~\cite{DCH:MultipliersOfTheFourierAlgebrasOfSomeSimpleLieGroupsAndTheirDiscreteSubgroups} De Canni{\`e}re and Haagerup introduced the Banach algebra $\MA(G)$ of \emph{Fourier multipliers} of $G$, consisting of functions $\varphi:\G\to\C$ such that
\begin{equation*}
	\varphi\psi\in\FA(\G)\qquad(\psi\in\FA(\G)),
\end{equation*}
where $\FA(\G)$ is the \emph{Fourier algebra} of $\G$ as introduced by Eymard in~\cite{Eym:L'alg`ebreDeFourierD'unGroupeLocalementCompact} (the \emph{Fourier--Stieltjes algebra} $\FSA(\G)$ of $\G$ is also introduced in this paper). The norm of $\varphi$ (denoted \emph{$\|\varphi\|_{\MA(\G)}$}) is given by considering $\varphi$ as an operator on $\FA(\G)$. According to~\cite[Proposition~1.2]{DCH:MultipliersOfTheFourierAlgebrasOfSomeSimpleLieGroupsAndTheirDiscreteSubgroups} a Fourier multiplier of $G$ can also be characterized as a continuous function $\varphi:\G\to\C$ such that
\begin{equation*}
	\lambda(g)\stackrel{M_\varphi}{\mapsto}\varphi(g)\lambda(g)\qquad(g\in\G)
\end{equation*}
extends to a $\sigma$-weakly continuous operator (still denoted $M_\varphi$) on the group von Neumann algebra ($\lambda:\G\to\linbeg(\ELL^2(\G))$ is the \emph{left regular representation} and the group von Neumann algebra is the closure of the span of $\lambda(\G)$ in the weak operator topology). Moreover, one has $\|\varphi\|_{\MA(\G)}=\|M_\varphi\|$. The Banach algebra $\MoA(\G)$ of \emph{completely bounded Fourier multipliers} of $G$ consists of the Fourier multipliers of $G$, $\varphi$, for which $M_\varphi$ is completely bounded. In this case they put \emph{$\|\varphi\|_{\MoA(\G)}=\|M_\varphi\|_{\cb}$}.

In~\cite{BF:Herz-SchurMultipliersAndCompletelyBoundedMultipliersOfTheFourierAlgebraOfALocallyCompactGroup} Bo{\.z}ejko and Fendler show that the completely bounded Fourier multipliers coincide isometrically with the continuous Herz--Schur multipliers. In~\cite{Jol:ACharacterizationOfCompletelyBoundedMultipliersOfFourierAlgebras} Jolissaint gives a short and self-contained proof of the result from~\cite{BF:Herz-SchurMultipliersAndCompletelyBoundedMultipliersOfTheFourierAlgebraOfALocallyCompactGroup} in the form stated below.
\begin{proposition}[\cite{BF:Herz-SchurMultipliersAndCompletelyBoundedMultipliersOfTheFourierAlgebraOfALocallyCompactGroup},~\cite{Jol:ACharacterizationOfCompletelyBoundedMultipliersOfFourierAlgebras}]
	\label{Gilbert0}
	Let $G$ be a locally compact group and assume that $\varphi:G\to\C$ and $k\geq0$ are given, then the following are equivalent:
	\begin{itemize}
		\item [(i)]$\varphi$ is a completely bounded Fourier multiplier of $\G$ with $\|\varphi\|_{\MoA(G)}\leq k$.
		\item [(ii)]$\varphi$ is a continuous Herz--Schur multiplier on $\G$ with $\|\varphi\|_{HS}\leq k$.
		\item [(iii)]There exists a Hilbert space $\Hil$ and two bounded, continuous maps $P,Q:G\to\Hil$ such that
		\begin{equation*}
			\varphi(y^{-1}x)=\ip{P(x)}{Q(y)}\qquad(x,y\in G)
		\end{equation*}
		and
		\begin{equation*}
			\|P\|_\infty\|Q\|_\infty\leq k,
		\end{equation*}
		where
		\begin{equation*}
			\|P\|_\infty=\sup_{x\in G}\|P(x)\|\quad\mbox{and}\quad\|Q\|_\infty=\sup_{y\in G}\|Q(y)\|.
		\end{equation*}
	\end{itemize}
\end{proposition}

By a \emph{representation} $(\pi,\Hil)$ of a locally compact group $\G$ on a Hilbert space $\Hil$ we mean a homomorphism of $\G$ into the invertible elements of $\linbeg(\G)$. A representation $(\pi,\Hil)$ of $\G$ is said to be \emph{uniformly bounded} if
\begin{equation*}
	\sup_{g\in\G}\|\pi(g)\|<\infty
\end{equation*}
and one usually writes $\|\pi\|$ for $\sup_{g\in\G}\|\pi(g)\|$. If $g\mapsto\pi(g)$ is continuous with respect to the strong operator topology on $\linbeg(\G)$ then we say that $(\pi,\Hil)$ is \emph{strongly continuous}.
Let $(\pi,\Hil)$ be a strongly continuous, uniformly bounded representation of $\G$ then, according to~\cite[Theorem~2.2]{DCH:MultipliersOfTheFourierAlgebrasOfSomeSimpleLieGroupsAndTheirDiscreteSubgroups}, any \emph{coefficient of $(\pi,\Hil)$} is a continuous Herz--Schur multiplier, i.e.,
\begin{equation*}
	g\stackrel{\varphi}{\mapsto}\ip{\pi(g)\xi}{\eta}\qquad(g\in\G)
\end{equation*}
is a continuous Herz--Schur multiplier with
\begin{equation*}
	\|\varphi\|_{\MoA(\G)}\leq\|\pi\|^2\|\xi\|\|\eta\|
\end{equation*}
for any $\xi,\eta\in\Hil$ (note that this result also follows as a corollary to Proposition~\ref{Gilbert0}).
U.~Haagerup has shown that on the non-abelian free groups there are Herz--Schur multipliers which can not be realized as coefficients of uniformly bounded representations. The proof by Haagerup has remained unpublished, but Pisier has later given a different proof, cf.~\cite{Pis:AreUnitarizableGroupsAmenable?}.
Haagerup's proof can be modified to prove the corresponding result for the connected, real rank one, simple Lie groups with finite center, cf.~\cite[Theorem~3.6]{Ste:CompletelyBoundedMultiplierNormOfSphericalFunctionsOnTheGeneralizedLorentzGroups}.

Strictly speaking, the requirement that the above representations be uniformly bounded is not fully needed in order to construct a continuous Herz--Schur multiplier. From Proposition~\ref{Gilbert0} it follows that it is enough to require that
\begin{equation}
	\label{unifbdreplacement}
	\sup_{x\in\G}\|\pi(x)\xi\|<\infty\quad\mbox{and}\quad\sup_{y\in\G}\|\pi(y^{-1})^*\eta\|<\infty.
\end{equation}
In~\cite[Theorem~1.1]{BF:Herz-SchurMultipliersAndUniformlyBoundedRepresentationsOfDiscreteGroups} Bo{\.z}ejko and Fendler shows that, for countable discrete groups, all Herz--Schur multipliers can be realized as coefficients of representations satisfying a condition similar to~\eqref{unifbdreplacement}. More specifically, they show that if $\varphi$ is a Hermitian Herz--Schur multiplier on a countable discrete group $\Gamma$, then there exists a representation $(\pi,\Hil)$ and vectors $\xi,\eta\in\Hil$ such that
\begin{equation*}
	\varphi(y^{-1}x)=\ip{\pi(x)\xi}{\pi(y)\eta}\qquad(x,y\in\Gamma),
\end{equation*}
with
\begin{equation*}
	\sup_{x\in\Gamma}\|\pi(x)\xi\|\leq\|\varphi\|_{\MoA(\Gamma)}^{\tfrac{1}{2}}\quad\mbox{and}\quad\sup_{y\in\Gamma}\|\pi(y)\eta\|\leq\|\varphi\|_{\MoA(\Gamma)}^{\tfrac{1}{2}}.
\end{equation*}
Furthermore, they give a quantitative bound on $\|\pi(g)\|$ for $g\in\Gamma$ and note that the same holds for non-Hermitian Herz--Schur multipliers by including a square root two factor in the bound for $\sup_{x\in\Gamma}\|\pi(x)\xi\|$ and $\sup_{y\in\Gamma}\|\pi(y)\eta\|$. In section~\ref{coeffOfNonUnifBd} we present a generalization of the result by Bo{\.z}ejko and Fendler to second countable, locally compact groups (cf.~Theorem~\ref{mt} and Corollary~\ref{mtc}). We also present the following theorem based on criterion~\eqref{unifbdreplacement} (cf.~Theorem~\ref{mt2}):
\begin{theorem}
	Let $\G$ be a second countable, locally compact group and let $\dd$ be a proper, left invariant metric on $\G$ which has \emph{at most exponential growth}, i.e.,
	\begin{equation*}
		\mu(B_n(e))\leq a\cdot e^{b n}\qquad(n\in\N)
	\end{equation*}
	for some constants $a,b>0$, where $\mu$ is a left invariant Haar measure on $\G$ and $B_n(e)=\setw{g\in\G }{ \dd(g,e)<n }$ is the open ball of radius $n$ centered around the identity element $e\in\G$. Then for any continuous Herz--Schur multiplier $\varphi$ on $\G$ there exists a strongly continuous representation $(\pi,\Hil)$ and vectors $\xi,\eta\in\Hil$ such that
	\begin{equation*}
		\varphi(y^{-1}x)=\ip{\pi(x)\xi}{\pi(y^{-1})^*\eta}\qquad(x,y\in\G),
	\end{equation*}
	with
	\begin{equation*}
		\sup_{x\in\G}\|\pi(x)\xi\|=\|\varphi\|_{\MoA(\G)}^{\tfrac{1}{2}}\quad\mbox{and}\quad\sup_{y\in\G}\|\pi(y^{-1})^*\eta\|=\|\varphi\|_{\MoA(\G)}^{\tfrac{1}{2}}.
	\end{equation*}
	Moreover, for every fixed $c>b$, $(\pi,\Hil)$ can be chosen such that
	\begin{equation*}
		\|\pi(g)\|\leq e^{\tfrac{c}{2}\cdot\dd(g,e)}\qquad(g\in\G).
	\end{equation*}
\end{theorem}
Note that the existence of a proper, left invariant metric with at most exponential growth on a second countable, locally compact group is guarantied by~\cite{HP:ProperMetricsOnLocallyCompactGroupsAndProperAffineIsometricActionsOnBanachSpaces}.

	\section{Coefficients of non-uniformly bounded representations}
\label{coeffOfNonUnifBd}
Second countability guarantees the existence of a proper, left invariant metric, cf.~\cite[Theorem]{Str:MetricsInLocallyCompactGroups}. Actually, according to Haagerup and Przybyszewska (cf.~\cite{HP:ProperMetricsOnLocallyCompactGroupsAndProperAffineIsometricActionsOnBanachSpaces}) one can choose this metric, $\dd$, to have \emph{at most exponential growth}, i.e.,
\begin{equation}
	\label{expgrowth}
	\mu(B_n(e))\leq a\cdot e^{b n}\qquad(n\in\N)
\end{equation}
for some constants $a,b>0$, where $\mu$ is a left invariant Haar measure on $\G$ and $B_n(e)=\setw{g\in\G }{ \dd(g,e)<n }$ is the open ball of radius $n$ centered around the identity element $e\in\G$.

Inspired by the proof of~\cite[Theorem~1.1]{BF:Herz-SchurMultipliersAndUniformlyBoundedRepresentationsOfDiscreteGroups}, we state and prove Theorem~\ref{mt} for \emph{Hermitian} Herz--Schur multipliers, i.e., Herz--Schur multipliers $\varphi$ for which $\varphi^*=\varphi$, where
\begin{equation*}
	\varphi^*(g)=\overline{\varphi(g^{-1})}\qquad(g\in\G).
\end{equation*}
The general (not necessarily Hermitian) case is treated in Corollary~\ref{mtc} and Theorem~\ref{mt2}.
\begin{theorem}
	\label{mt}
	If $\varphi$ is a continuous Hermitian Herz--Schur multiplier on a second countable, locally compact group $\G$, and $\dd$ is a proper, left invariant metric on $\G$ satisfying~\eqref{expgrowth} for some $a,b>0$ (which exists according to~\cite{HP:ProperMetricsOnLocallyCompactGroupsAndProperAffineIsometricActionsOnBanachSpaces}), then there exists a strongly continuous representation $(\pi,\Hil)$ and vectors $\xi,\eta\in\Hil$ such that
	\begin{equation*}
		\varphi(y^{-1}x)=\ip{\pi(x)\xi}{\pi(y)\eta}\qquad(x,y\in\G),
	\end{equation*}
	with
	\begin{equation*}
		\sup_{x\in\G}\|\pi(x)\xi\|=\|\varphi\|_{\MoA(\G)}^{\tfrac{1}{2}}\quad\mbox{and}\quad\sup_{y\in\G}\|\pi(y)\eta\|=\|\varphi\|_{\MoA(\G)}^{\tfrac{1}{2}}.
	\end{equation*}
	Moreover, for every fixed $c>b$, $(\pi,\Hil)$ can be chosen such that
	\begin{equation*}
		\|\pi(g)\|\leq\ e^{\tfrac{c}{2}\cdot\dd(g,e)}\qquad(g\in\G).
	\end{equation*}
\end{theorem}
Before we proceed with the proof of Theorem~\ref{mt} we need the following application of the above mentioned result from~\cite{HP:ProperMetricsOnLocallyCompactGroupsAndProperAffineIsometricActionsOnBanachSpaces}, which was communicated to us by Haagerup.
\begin{lemma}
	\label{lemma}
	If $\G$ is a second countable, locally compact group, then there exists a positive function $h\in\ELL^1(\G)$ with $\|h\|_1=1$, and a positive function $\cc$ on $\G$ such that
	\begin{equation*}
		\frac{1}{\cc(g)}\int_\G f(z)h(z)\dd\mu(z)\leq\int_\G f(z)h(g z)\dd\mu(z)\leq\cc(g)\int_\G f(z)h(z)\dd\mu(z)
	\end{equation*}
	for $g\in\G$ and any positive $f\in\ELL^\infty(\G)$, where $\mu$ is the Haar measure on $\G$. Moreover, we may use
	\begin{equation*}
		h(g)=\frac{e^{-c\cdot\dd(g,e)}}{\int_\G e^{-c\cdot\dd(x,e)}\dd\mu(x)}\quad\mbox{and}\quad\cc(g)=e^{c\cdot\dd(g,e)}\qquad(g\in\G)
	\end{equation*}
	for $c>b$ when $\dd$ is a proper, left invariant metric on $\G$ satisfying~\eqref{expgrowth}.
\end{lemma}
\begin{proof}
	Let $\mu$ be a left invariant Haar measure on $\G$ and let $\dd$ be a proper, left invariant metric on $\G$ satisfying~\eqref{expgrowth}. We claim that
	\begin{equation*}
		0<\int_\G e^{-c\cdot\dd(g,e)}\dd\mu(g)<\infty.
	\end{equation*}
	 Put $E_1=B_1(e)$ and define inductively
	\begin{equation*}
		E_n=B_n(e)\setminus B_{n-1}(e)\qquad(n\geq2).
	\end{equation*}
	Then $\G=\bigsqcup_{n=1}^\infty E_n$ and
	\begin{equation*}
		e^{-c n}\leq e^{-c\cdot\dd(g,e)}\leq e^{-c(n-1)}\qquad(g\in E_n).
	\end{equation*}
	Hence,
	\begin{eqnarray*}
		\int_\G e^{-c\cdot\dd(g,e)}\dd\mu(g) & = & \sum_{n=1}^\infty\int_{E_n} e^{-c\cdot\dd(g,e)}\dd\mu(g)\\
		& \leq & \sum_{n=1}^\infty e^{-c(n-1)} \mu(E_n)\\
		& \leq & e^c\sum_{n=1}^\infty e^{-c n} \mu(B_n(e))\\
		& \leq & a e^c\sum_{n=1}^\infty e^{(b-c)n}\\
		& < & \infty
	\end{eqnarray*}
	because $c>b$.
	
	By the reverse triangle inequality we see that
	\begin{equation*}
		|\dd(z,g^{-1})-\dd(z,e)|\leq\dd(e,g^{-1})\qquad(g,z\in\G).
	\end{equation*}
	Using left invariance of the metric one finds that
	\begin{equation*}
		|\dd(g z,e)-\dd(z,e)|\leq\dd(g,e)\qquad(g,z\in\G).
	\end{equation*}
	This implies
	\begin{equation*}
		\frac{1}{\cc(g)}e^{-c\cdot\dd(z,e)}\leq e^{-c\cdot\dd(g z,e)}\leq\cc(g)e^{-c\cdot\dd(z,e)}\qquad(g,z\in\G),
	\end{equation*}
	which is easily seen to complete the proof.
\end{proof}
\begin{lemma}
	\label{newrep}
	Assume that $\Hil$ is a Hilbert space and that $R:G\to\Hil$ is bounded and continuous. Let $R':G\to\ELL^2(\G,\Hil,\mu)$ be given by
	\begin{equation*}
		R'(x)(z)=\sqrt{h(z)}R(z^{-1}x)\qquad(z\in\G)
	\end{equation*}
	for $x\in\G$, then $R'$ is bounded and continuous, with $\|R'(x)\|_2\leq\|R\|_\infty$ for all $x\in\G$. Also, let $\Cpt_R=\overline{\sspan}\setw{R'(x)}{x\in\G}$ be a sub-Hilbert space of $\ELL^2(\G,\Hil,\mu)$. Then there exists a unique representation $(\pi_R,\Cpt_R)$ such that
	\begin{equation*}
		\pi_R(g)R'(x)=R'(g x)\qquad(g,x\in\G).
	\end{equation*}
	Moreover,
	\begin{equation*}
		\|\pi_R(g)\|\leq e^{\tfrac{c}{2}\cdot\dd(g,e)}\qquad(g\in\G)
	\end{equation*}
	and the representation is strongly continuous.
\end{lemma}
\begin{proof}
	From Lebesgue's dominated convergence theorem it easily follows that $R'$ is continuous. To see that $R'$ is bounded, note that
	\begin{equation*}
		\|R'(x)\|_2^2=\int_\G h(z)\|R(z^{-1}x)\|^2\dd\mu(z)\leq\|R\|_\infty^2\qquad(x\in\G).
	\end{equation*}
	
	If $n\in\N$, $x_1,\ldots,x_n\in\G$ and $c_1,\ldots,c_n\in\C$, then Lemma~\ref{lemma} implies that
	\begin{equation*}
		\int_\G\|\sum_{i=1}^n c_iR(z^{-1}x_i)\|^2h(g z)\dd\mu(z)\leq\cc(g)\int_\G\|\sum_{i=1}^n c_iR(z^{-1}x_i)\|^2h(z)\dd\mu(z)
	\end{equation*}
	for $g\in\G$, where
	\begin{equation*}
		\cc(g)=e^{c\cdot\dd(g,e)}\qquad(g\in\G).
	\end{equation*}
	It follows that
	\begin{equation*}
		\|\sum_{i=1}^n c_iR'(g x_i)\|_2^2\leq\cc(g)\|\sum_{i=1}^n c_iR'(x_i)\|_2^2\qquad(g\in\G),
	\end{equation*}
	from which we conclude that there exists a unique representation $(\pi_R,\Cpt_R)$ of $\G$ such that
	\begin{equation*}
		\pi_R(g)R'(x)=R'(g x)\qquad(g,x\in\G).
	\end{equation*}
	Furthermore,
	\begin{equation*}
		\|\pi_R(g)\|\leq\sqrt{\cc(g)}\qquad(g\in\G).
	\end{equation*}
	We proceed to show that the representation is strongly continuous. Since $\sspan\setw{R'(x) }{ x\in\G}$ is total in $\Cpt_R$ and $\|\pi_R(g)\|\leq\sqrt{\cc(g)}$, where $g\mapsto\sqrt{\cc(g)}$ is a continuous function, it is enough to show that
	\begin{equation*}
		\lim_{n\to\infty}\pi_R(g_n)R'(x)=R'(x)
	\end{equation*}
	for $x\in\G$, when $(g_n)_{n\in\N}$ is a sequence converging to the identity $e\in\G$ (since $\G$ is second countable we do not have to consider nets). But $\pi_R(g_n)R'(x)=R'(g_n x)$ so this follows simply from continuity of $R'$.
\end{proof}
\begin{proof}[Proof of Theorem~\ref{mt}]
	Assume that $\varphi$ is a continuous Hermitian Herz--Schur multiplier and use Proposition~\ref{Gilbert0} to find a Hilbert space $\Hil$ and bounded, continuous maps $P,Q:\G\to\Hil$ such that
	\begin{equation*}
		\varphi(y^{-1}x)=\ip{P(x)}{Q(y)}\qquad(x,y\in G)
	\end{equation*}
	and
	\begin{equation*}
		\|P\|_\infty=\|Q\|_\infty=\|\varphi\|_{\MoA(\G)}^{\tfrac{1}{2}}.
	\end{equation*}
	Define
	\begin{equation*}
		a_\pm(x,y)=\frac{1}{4}\ip{P(x)\pm Q(x)}{P(y)\pm Q(y)}\qquad(x,y\in\G).
	\end{equation*}
	This gives rise to two positive definite, bounded kernels on $\G\times\G$ satisfying
	\begin{equation*}
		\label{1'}
		a_+(x,y)-a_-(x,y)=\frac{1}{2}\varphi(y^{-1}x)+\frac{1}{2}\varphi^*(y^{-1}x)=\varphi(y^{-1}x)\qquad(x,y\in\G)
	\end{equation*}
	and
	\begin{equation*}
		\label{3'}
		a_+(x,x)+a_-(x,x)=\frac{1}{2}\|P(x)\|^2+\frac{1}{2}\|Q(x)\|^2\leq\|\varphi\|_{\MoA(\G)}\qquad(x\in\G).
	\end{equation*}

	Let
	\begin{equation*}
		h(g)=\frac{e^{-c\cdot\dd(g,e)}}{\int_\G e^{-c\cdot\dd(x,e)}\dd\mu(x)}\qquad(g\in\G)
	\end{equation*}
	for some $c>b$, when $\dd$ is a proper, left invariant metric on $\G$ satisfying~\eqref{expgrowth} (cf.~Lemma~\ref{lemma}). Define $(P\pm Q)':\G\to\ELL^2(\G,\Hil,\mu)$ by
	\begin{equation*}
		(P\pm Q)'(x)(z)=\sqrt{h(z)}(P\pm Q)(z^{-1}x)\qquad(z\in\G)
	\end{equation*}
	for $x\in\G$. According to Lemma~\ref{newrep} there exist strongly continuous representations $(\pi_{P\pm Q},\Cpt_{P\pm Q})$, where $\Cpt_{P\pm Q}=\overline{\sspan}\setw{(P'\pm Q')(x)}{x\in\G}$ and $\pi_{P\pm Q}(g)(P\pm Q)'(x)=(P\pm Q)'(g x)$ for $g,x\in\G$. Furthermore, these representations satisfy
	\begin{equation}
		\label{claim}
		\|\pi_{P\pm Q}(g)\|\leq e^{\tfrac{c}{2}\cdot\dd(g,e)}\qquad(g\in\G).
	\end{equation}
	Put
	\begin{equation*}
		A_\pm(x,y)=\ip{(P\pm Q)'(x)}{(P\pm Q)'(y)}_{\Cpt_{P\pm Q}}\qquad(x,y\in\G).
	\end{equation*}
	Then $A_\pm$ are positive definite, bounded kernels on $\G\times\G$ satisfying
	\begin{equation}
		\label{1Coeff}
		A_+(x,y)-A_-(x,y)=\varphi(y^{-1}x)\qquad(x,y\in\G)
	\end{equation}
	and
	\begin{equation}
		\label{3}
		A_+(x,x)+A_-(x,x)\leq\|\varphi\|_{\MoA(\G)}\qquad(x\in\G).
	\end{equation}
	To make the notation less cumbersome, let $\pi_\pm=\pi_{P\pm Q}$ and $\Cpt_\pm=\Cpt_{P\pm Q}$ and define $\xi_\pm=(P\pm Q)'(e)$. Notice that
	\begin{equation*}
		\ip{\pi_\pm(x)\xi_\pm}{\pi_\pm(y)\xi_\pm}_{\Cpt_\pm}=A_\pm(x,y)\qquad(x,y\in\G),
	\end{equation*}
	and that~\eqref{claim} now reads
	\begin{equation*}
		\|\pi_\pm(g)\|\leq e^{\tfrac{c}{2}\cdot\dd(g,e)}\qquad(g\in\G).
	\end{equation*}
	Put
	\begin{equation*}
		\Cpt=\Cpt_+\oplus\Cpt_-,\quad\xi=\xi_+\oplus\xi_-,\quad\eta=\xi_+\oplus-\xi_-\quad\mbox{and}\quad\pi=\pi_+\oplus\pi_-.
	\end{equation*}
	Observe that $\pi$ is a strongly continuous representation such that
	\begin{equation}
		\label{pibd}
		\|\pi(g)\|\leq e^{\tfrac{c}{2}\cdot\dd(g,e)}\qquad(g\in\G)
	\end{equation}
	and
	\begin{equation*}
		\ip{\pi(x)\xi}{\pi(y)\eta}_\Cpt=\varphi(y^{-1}x)\qquad(x,y\in\G).
	\end{equation*}
	Finally, observe that
	\begin{equation*}
		\|\pi(x)\xi\|^2=\|\pi_+(x)\xi_+\|^2+\|\pi_-(x)\xi_-\|^2=A_+(x,x)+A_-(x,x)\leq\|\varphi\|_{\MoA(\G)}
	\end{equation*}
	for $x\in\G$, and similarly
	\begin{equation*}
		\|\pi(y)\eta\|^2\leq\|\varphi\|_{\MoA(\G)}
	\end{equation*}
	for $y\in\G$. This finishes the proof.
\end{proof}
\begin{corollary}
	\label{mtc}
	If $\varphi$ is a continuous Herz--Schur multiplier on a second countable, locally compact group $\G$, and $\dd$ is a proper, left invariant metric on $\G$ satisfying~\eqref{expgrowth} for some $a,b>0$ (which exists according to~\cite{HP:ProperMetricsOnLocallyCompactGroupsAndProperAffineIsometricActionsOnBanachSpaces}), then there exists a strongly continuous representation $(\pi,\Hil)$ and vectors $\xi,\eta\in\Hil$ such that
	\begin{equation*}
		\varphi(y^{-1}x)=\ip{\pi(x)\xi}{\pi(y)\eta}\qquad(x,y\in\G),
	\end{equation*}
	with
	\begin{equation*}
		\sup_{x\in\G}\|\pi(x)\xi\|\leq\sqrt{2}\|\varphi\|_{\MoA(\G)}^{\tfrac{1}{2}}\quad\mbox{and}\quad\sup_{y\in\G}\|\pi(y)\eta\|\leq\sqrt{2}\|\varphi\|_{\MoA(\G)}^{\tfrac{1}{2}}.
	\end{equation*}
	Moreover, for every fixed $c>b$, $(\pi,\Hil)$ can be chosen such that
	\begin{equation*}
		\|\pi(g)\|\leq e^{\tfrac{c}{2}\cdot\dd(g,e)}\qquad(g\in\G).
	\end{equation*}
\end{corollary}
\begin{proof}
	This follows from Theorem~\ref{mt} since
	\begin{equation*}
		\varphi=\Re(\varphi)+i\Im(\varphi),
	\end{equation*}
	where
	\begin{equation*}
		\Re(\varphi)=\frac{\varphi+\varphi^*}{2}\quad\mbox{and}\quad\Im(\varphi)=\frac{\varphi-\varphi^*}{2i}
	\end{equation*}
	are continuous Hermitian Herz--Schur multipliers with
	\begin{equation*}
		\|\Re(\varphi)\|_{\MoA(\G)}\leq\|\varphi\|_{\MoA(\G)}\quad\mbox{and}\quad\|\Im(\varphi)\|_{\MoA(\G)}\leq\|\varphi\|_{\MoA(\G)}.
	\end{equation*}
\end{proof}
\begin{maintheorem}
	\label{mt2}
	If $\varphi$ is a continuous Herz--Schur multiplier on a second countable, locally compact group $\G$, and $\dd$ is a proper, left invariant metric on $\G$ satisfying~\eqref{expgrowth} for some $a,b>0$ (which exists according to~\cite{HP:ProperMetricsOnLocallyCompactGroupsAndProperAffineIsometricActionsOnBanachSpaces}), then there exists a strongly continuous representation $(\pi,\Hil)$ and vectors $\xi,\eta\in\Hil$ such that
	\begin{equation*}
		\varphi(y^{-1}x)=\ip{\pi(x)\xi}{\pi(y^{-1})^*\eta}\qquad(x,y\in\G),
	\end{equation*}
	with
	\begin{equation*}
		\sup_{x\in\G}\|\pi(x)\xi\|=\|\varphi\|_{\MoA(\G)}^{\tfrac{1}{2}}\quad\mbox{and}\quad\sup_{y\in\G}\|\pi(y^{-1})^*\eta\|=\|\varphi\|_{\MoA(\G)}^{\tfrac{1}{2}}.
	\end{equation*}
	Moreover, for every fixed $c>b$, $(\pi,\Hil)$ can be chosen such that
	\begin{equation*}
		\|\pi(g)\|\leq e^{\tfrac{c}{2}\cdot\dd(g,e)}\qquad(g\in\G).
	\end{equation*}
\end{maintheorem}
\begin{proof}
	Assume that $\varphi$ is a continuous Herz--Schur multiplier and use Proposition~\ref{Gilbert0} to find a Hilbert space $\Hil$ and bounded, continuous maps $P,Q:\G\to\Hil$ such that
	\begin{equation*}
		\varphi(y^{-1}x)=\ip{P(x)}{Q(y)}\qquad(x,y\in G)
	\end{equation*}
	and
	\begin{equation*}
		\|P\|_\infty=\|Q\|_\infty=\|\varphi\|_{\MoA(\G)}^{\tfrac{1}{2}}.
	\end{equation*}
	Let
	\begin{equation*}
		h(g)=\frac{e^{-c\cdot\dd(g,e)}}{\int_\G e^{-c\cdot\dd(x,e)}\dd\mu(x)}\qquad(g\in\G)
	\end{equation*}
	for some $c>b$, when $\dd$ is a proper, left invariant metric on $\G$ satisfying~\eqref{expgrowth} (cf.~Lemma~\ref{lemma}). Define $P',Q':\G\to\ELL^2(\G,\Hil,\mu)$ by
	\begin{equation*}
		P'(x)(z)=\sqrt{h(z)}P(z^{-1}x)\quad\mbox{and}\quad Q'(y)(z)=\sqrt{h(z)}Q(z^{-1}y)\qquad(z\in\G)
	\end{equation*}
	for $x,y\in\G$. According to Lemma~\ref{newrep} there exists a strongly continuous representation $(\pi_P,\Cpt_P)$, where $\Cpt_P=\overline{\sspan}\setw{P'(x)}{x\in\G}$ and $\pi_P(g)P'(x)=P'(g x)$ for $g,x\in\G$. Furthermore, this representation satisfies
	\begin{equation*}
		\|\pi_P(g)\|\leq e^{\tfrac{c}{2}\cdot\dd(g,e)}\qquad(g\in\G).
	\end{equation*}
	Observe that
	\begin{equation*}
		\|P'(x)\|_2^2,\|Q'(y)\|_2^2\leq\|\varphi\|_{\MoA(\G)}
	\end{equation*}
	and
	\begin{equation*}
		\ip{P'(x)}{Q'(y)}_{\ELL^2(\G,\Hil,\mu)}=\int_\G h(z)\ip{P(z^{-1}x)}{Q(z^{-1}y)}_\Hil\dd\mu(z)=\varphi(y^{-1}x)
	\end{equation*}
	for $x,y\in\G$. Put $\xi=P'(e)$ and $\eta=P_{\Cpt_P} Q'(e)$, where $P_{\Cpt_P}$ is the orthogonal projection on $\Cpt_P$. Note that $\xi,\eta\in\Cpt_P$ and
	\begin{equation*}
		\varphi(y^{-1}x)=\ip{\pi_P(y^{-1}x)\xi}{\eta}_{\Cpt_P}=\ip{\pi_P(x)\xi}{\pi_P(y^{-1})^*\eta}_{\Cpt_P}\qquad(x,y\in\G).
	\end{equation*}
	It is clear that $\|\pi_P(x)\xi\|_{\Cpt_P}^2=\|P'(x)\|_2^2\leq\|\varphi\|_{\MoA(\G)}$. The corresponding result for $\|\pi_P(y^{-1})^*\eta\|_{\Cpt_P}^2$ requires more work. For $x\in\G$ arbitrary we find that
	\begin{eqnarray*}
		\ip{\pi_P(y^{-1})P'(x)}{P_{\Cpt_P} Q'(e)}_{\Cpt_P} & = & \ip{P'(y^{-1}x)}{P_{\Cpt_P} Q'(e)}_{\Cpt_P}\\
		& = & \ip{P'(y^{-1}x)}{Q'(e)}_\Hil\\
		& = & \varphi(y^{-1}x)\\
		& = & \ip{P'(x)}{Q'(y)}_\Hil\\
		& = & \ip{P'(x)}{P_{\Cpt_P} Q'(y)}_{\Cpt_P},
	\end{eqnarray*}
	from which we conclude that $\pi_P(y^{-1})^*P_{\Cpt_P} Q'(e)=P_{\Cpt_P} Q'(y)$ and therefore
	\begin{equation*}
		\|\pi_P(y^{-1})^*\eta\|_{\Cpt_P}^2=\|P_{\Cpt_P} Q'(y)\|_{\Cpt_P}^2\leq\|Q'(y)\|_2^2\leq\|\varphi\|_{\MoA(\G)}.
	\end{equation*}
\end{proof}
\begin{remark}
	For the free group on $N$ generators ($2\leq N<\infty$) the constants $a,b$ in~\eqref{expgrowth} may be chosen as $a=\frac{N}{(N-1)(2N-1)}$ and $b=\ln(2N-1)$. This implies that for $r>\sqrt{2N-1}$, the representations $(\pi,\Hil)$ from Theorem~\ref{mt}, Corollary~\ref{mtc} and Theorem~\ref{mt2} may be chosen to satisfy $\|\pi(g)\|\leq r^{\dd(g,e)}$ for all $g\in\G$.
\end{remark}

	\subsection*{Acknowledgments}
I would like to thank Ryszard Szwarc and Marek Bo{\.z}ejko for an inspiring stay in Wroc{\l}aw during March 2009.

	\bibliography{troelsBibliography}
	\contrib{Troels Steenstrup}{troelssj@imada.sdu.dk}\smallskip\\
Department of Mathematics and Computer Science, University of Southern Denmark, Campusvej~55, DK--5230~Odense~M, Denmark.\\

\end{document}